\documentclass[reqno]{amsart}
\usepackage{amsmath,amssymb,amsthm,enumerate,hyperref,indentfirst}

\usepackage{geometry}
\newtheorem{theorem}{Theorem}[section]
\newtheorem{proposition}[theorem]{Proposition}
\newtheorem{lemma}[theorem]{Lemma}
\newtheorem{corollary}[theorem]{Corollary}

\theoremstyle{definition}
\newtheorem{definition}[theorem]{Definition}

\theoremstyle{remark}
\newtheorem{remark}[theorem]{Remark}

\numberwithin{equation}{section}

\newcommand{\ab}{\mathrm{ab}}
\newcommand{\divisible}{\mathrm{div}}
\newcommand{\fin}{\mathrm{f}}
\newcommand{\Gal}{\mathrm{Gal}}
\newcommand{\Hom}{\mathrm{Hom}}
\newcommand{\separable}{\mathrm{sep}}
\newcommand{\tor}{\mathrm{tor}}

\DeclareMathOperator{\ord}{ord}

\begin{document}

\title{Kummer-faithfulness for function fields}
\author{Takuya Asayama}
\address[Takuya Asayama]{Department of Mathematics, Tokyo Institute of Technology, 2-12-1, O-okayama, Meguro-ku, Tokyo 152-8551, Japan}
\email{asayama.t.aa@m.titech.ac.jp}
\date{}
\maketitle

\begin{abstract}
A perfect field $K$ is said to be Kummer-faithful if the Mordell-Weil group of every semi-abelian variety over every finite extension of $K$ has no nonzero divisible element.
The class of Kummer-faithful fields contains that of sub-$p$-adic fields and is thought to be suitable for developing anabelian geometry.
In this paper, we investigate a function field analogue of the notion of Kummer-faithful fields.
We introduce a notion of Drinfeld-Kummer-faithful (DKF) fields using Drinfeld modules.
A sufficient condition for a Galois extension of a function field to be DKF is provided in terms of ramification theory.
More precisely, a Galois extension with finite maximal ramification break outside the infinite prime $(1 / t)$ over a finite extension of the rational function field $\mathbb{F}_q(t)$ over the finite field $\mathbb{F}_q$ of $q$ elements is DKF.
Some examples of DKF fields are also given.
The construction of these examples is inspired by Ozeki and Taguchi's examples of highly Kummer-faithful fields.
\end{abstract}

\section{Introduction}

Kummer-faithful fields, introduced by Mochizuki~\cite{Mochizuki}, are thought to be suitable for the base field in anabelian geometry.
They are defined by the triviality of the divisible parts of the Mordell-Weil groups of semi-abelian varieties as follows.

\begin{definition}
Let $K$ be a perfect field.
We say that $K$ is \textit{Kummer-faithful} if, for every finite extension $L$ of $K$ and every semi-abelian variety $X$ over $L$, it holds that $\bigcap_{n > 0} n \cdot X(L) = 0$, where $n$ runs over all positive rational integers.
\end{definition}

For instance, every sub-$p$-adic field, i.e., every field isomorphic to a subfield of a finitely generated extension of the $p$-adic field $\mathbb{Q}_p$, is Kummer-faithful.
There is also a Kummer-faituful field which is not sub-$p$-adic~\cite[Remark~1.5.4]{Mochizuki}.

Ozeki and Taguchi~\cite{OT} introduced the notion of highly Kummer-faithful fields and studied its properties in terms of ramification theory.
The notion of highly Kummer-faithful fields is a specialization of that of Kummer-faithful fields in the sense that high Kummer-faithfulness implies Kummer-faithfulness for a Galois extension of a Kummer-faithful field of characteristic zero~\cite[Proposition~2.8]{OT}.
They also provided some examples of highly Kummer-faithful fields.
One of their examples is constructed by adjoining to a number field $K$ the torsion points of all semi-abelian varieties $X$ over $K$ of bounded dimension, having order of bounded prime power.
The precise statement is the following.

\begin{theorem}[{\cite[Theorem~3.3]{OT}}]
Let $K$ be a number field, $g$ a positive integer, and $\boldsymbol{m} = (m_p)_p$ a family of non-negative integers, where $p$ runs over all rational prime numbers.
Let $K_{g, \boldsymbol{m}}$ be the extension of $K$ generated by all coordinates of elements of $X[p^{m_p}]$ for all semi-abelian varieties $X$ over $K$ of dimension at most $g$ and all rational prime number $p$.
Then $K_{g, \boldsymbol{m}}$ is highly Kummer-faithful.
In particular, it is Kummer-faithful.
\end{theorem}

This paper investigates a function field analogue of the notion of Kummer-faithful fields.
We introduce Drinfeld-Kummer-faithful (DKF) fields by the triviality of the divisible parts of Drinfeld modules.
A sufficient condition for a Galois extension of a function field to be DKF is provided in terms of ramification theory.
Some examples of DKF fields are also given.
One of these is constructed using torsion points of Drinfeld modules in a similar way as the above theorem, except that we add another parameter to limit the badness of reduction.
We note that the maximal break of the $t$-torsion points of a Drinfeld $\mathbb{F}_q[t]$-module $\phi$ over $\mathbb{F}_q(t)$ at a finite prime may be arbitrarily large even if $\phi$ is of rank~$2$ (see~\cite{AH}).

\section{Drinfeld-Kummer-faithful fields}

Let $F$ be a global function field over the finite field $\mathbb{F}_q$ of $q$ elements.
Fix a prime $\infty$ of $F$ and let $A$ be the subring of $F$ consisting of functions regular outside $\infty$.
For any field $E$, fix an algebraic closure $\overline{E}$ of $E$ and let $E^\separable$ be the separable closure of $E$ in $\overline{E}$.
Denote by $G_E$ the absolute Galois group $\Gal(E^\separable / E)$ of $E$.

In this section, we define a notion of Drinfeld-Kummer-faithful fields and study its properties.
Unlike that of Kummer-faithfulness, we do not assume that the field is perfect.
Throughout the present paper, we only consider Drinfeld modules of generic characteristic.

\begin{definition}
Let $M$ be an $A$-module.
An element $x$ in $M$ is called \textit{divisible} if, for every nonzero $a \in A$, there exists $y$ in $M$ such that $x = a y$.
For an ideal $\mathfrak{a}$ of $A$, we say that $x \in M$ is \textit{$\mathfrak{a}$-divisible} if, for every positive integer $n$, there exist $a \in \mathfrak{a}^n$ and $y \in M$ such that $x = a y$.
Denote by $M_\divisible$ (resp.\ $M_{\mathfrak{a}\text{-}\divisible}$) the set of divisible (resp.\ $\mathfrak{a}$-divisible) elements in $M$, i.e.,
\[ M_\divisible = \bigcap_{a \in A \setminus \{0\}} a M \quad \left(\text{resp.\ } M_{\mathfrak{a}\text{-}\divisible} = \bigcap_{n > 0} \mathfrak{a}^n M \right). \]
For simplicity, if $\mathfrak{a} = (a)$ is a principal ideal, we say $a$-divisible for $(a)$-divisible and write $M_{a\text{-}\divisible}$ for $M_{(a)\text{-}\divisible}$.
\end{definition}

\begin{definition}
    An extension $K$ over $F$ is called \textit{Drinfeld-Kummer-faithful} (we abbreviate it as \textit{DKF}) if, for every finite extension $L$ of $K$ and every Drinfeld $A$-module $\phi$ over $L$, it holds that $\phi(L)_\divisible = 0$.
    Here, we denote by $\phi(L)$ the set $L$ equipped with the $A$-module structure via $\phi$.
\end{definition}

It immediately follows from the definition that any intermediate field of a DKF field $K / F$ is also DKF.
Let $K'$ be a finite extension of $K / F$.
Then $K'$ is DKF if and only if $K$ is DKF.

If $K$ is a finitely generated field over $F$, then, for any Drinfeld $A$-module $\phi$ over $K$, the $A$-module $\phi(K)$ is the direct sum of a finite torsion module and a free $A$-module of rank $\aleph_0$~\cite{Poonen, Wang}, hence we have $\phi(K)_\divisible = 0$.
Therefore $K$ is DKF.
On the other hand, the algebraic closure $\overline{F}$ and the separable closure $F^\separable$ of $F$ are obviously not DKF.
The DKF-ness represents that the field is small in the sense that Drinfeld modules have no nonzero divisible points.

Let $\phi$ be a Drinfeld $A$-module and $\mathfrak{p}$ be a nonzero prime ideal.
The \textit{$\mathfrak{p}$-adic Tate module} of $\phi$ is defined by $T_\mathfrak{p}(\phi) = \varprojlim_n \phi[\mathfrak{p}^n]$.
The following proposition is an analogue of \cite[Proposition~2.4]{OT}.

\begin{proposition}\label{criteria}
Let $K \subseteq L$ be two extensions in $\overline{F}$ over $F$.
Let $\phi$ be a Drinfeld $A$-module over $K$.
\begin{enumerate}[\textup{(\arabic{enumi})}]
    \item For any nonzero prime ideal $\mathfrak{p}$ in $A$, the following conditions are equivalent$:$
    \begin{enumerate}[\textup{(\roman{enumii})}]
        \item $(\phi(L)[\mathfrak{p}^\infty])_{\mathfrak{p}\text{-}\divisible} = 0$.
        \item $\phi(L)[\mathfrak{p}^\infty]$ is finite.
        \item $T_\mathfrak{p}(\phi)^{G_L} = 0$.
    \end{enumerate}
    \item Consider the following conditions on $\phi$$:$
    \begin{enumerate}[\textup{(\alph{enumii})}]
        \item $\phi(L)_\divisible = 0$.
        \item $(\phi(L)_\tor)_\divisible = 0$.
        \item $\phi(L)[\mathfrak{p}^\infty]$ is finite for any nonzero prime ideal $\mathfrak{p}$ in $A$.
    \end{enumerate}
    Then we have \textup{(a)} $\Rightarrow$ \textup{(b)} $\Leftrightarrow$ \textup{(c)}.
    If $K$ is DKF and $L / K$ is Galois, then we have \textup{(a)} $\Leftrightarrow$ \textup{(b)} $\Leftrightarrow$ \textup{(c)}.
\end{enumerate}
\end{proposition}

\begin{proof}
(1) The definition of Tate modules implies the equivalence of (ii) and (iii).
Since the class number of $A$ is finite (see~\cite[Lemma~4.1.2]{FJ}), there exists $a \in A$ such that the principal ideal $(a)$ is a $\mathfrak{p}$-power.
Then $T_\mathfrak{p}(\phi) = \varprojlim_n \phi[\mathfrak{p}^n] \cong \varprojlim_m \phi[a^m]$ and this isomorphism is compatible with the action of $G_L$.
The natural isomorphism $(\varprojlim_m \phi[a^m])^{G_L} \cong \Hom_A(A[1 / a] / A, \phi(L)[a^\infty])$ and $(\phi(L)[a^\infty])_{a\text{-}\divisible} = (\phi(L)[\mathfrak{p}^\infty])_{\mathfrak{p}\text{-}\divisible}$ yield the equivalence of (i) and (ii).

(2) It is obvious that (a) implies (b).
To see that (b) implies (c), we take $a_\mathfrak{p} \in A$ such that the principal ideal $(a_\mathfrak{p})$ is a $\mathfrak{p}$-power as in (1) for each $\mathfrak{p}$.
Then we have the natural isomorphism
\begin{align*}
    \prod_\mathfrak{p} T_\mathfrak{p}(\phi)^{G_L} &\cong \prod_\mathfrak{p} \Hom_A(A[1 / a_\mathfrak{p}] / A, \phi(L)[a_\mathfrak{p}^\infty]) \\
    &\cong \prod_\mathfrak{p} \Hom_A(A[1 / a_\mathfrak{p}] / A, \phi(L)_\tor) \\
    &\cong \Hom_A\left(\bigoplus_\mathfrak{p} A[1 / a_\mathfrak{p}] / A, \phi(L)_\tor\right) \\
    &\cong \Hom_A(F / A, \phi(L)_\tor),
\end{align*}
where $\mathfrak{p}$ runs over the nonzero prime ideals of $A$.
By (b), we have $T_\mathfrak{p}(\phi)^{G_L} = 0$ for all $\mathfrak{p}$.
Thus (c) holds from (1).

To show that (c) implies (b), we prove the contraposition.
Suppose that there exists a nonzero $x \in (\phi(L)_\tor)_\divisible$.
Let $\mathfrak{a}$ be the annihilator of $x$ and $\mathfrak{p}_1, \dotsc, \mathfrak{p}_n$ the prime ideals appearing in the prime decomposition of $\mathfrak{a}$.
Set $S = A \setminus \left(\bigcup \mathfrak{p}_i\right)$ and $B = S^{- 1} A$.
Then $B$ is a principal ideal domain since it is a Dedekind domain with only finitely many prime ideals (see~\cite[Chapter~I, Section~3, Exercise~4]{Neukirch}).
Hence there is $b \in B$ with $S^{- 1} \mathfrak{a} = (b)$.
In addition, $x$ is a nonzero divisible point in $S^{- 1} \phi(L)_\tor$ and $(b)$ is the annihilator of $x$ in $S^{- 1} \phi(L)_\tor$.
Let $\mathfrak{P} = (\pi)$ be a prime ideal of $B$ dividing $b$.
By renumbering the indices, we may assume $A \cap \mathfrak{P} = \mathfrak{p}_1$.
Then it follows that $\pi^{- 1} b x$ is a nonzero point in $(S^{- 1} \phi(L))[\mathfrak{P}^\infty]_{\mathfrak{P}\text{-}\divisible}$.
Take $s \in S$ satisfying $s \pi^{- 1} b \in A$ and let $y = s \pi^{- 1} b x$.
Then $y$ is a nonzero point in $\phi(L)[\mathfrak{p}_1^\infty]_{\mathfrak{p}_1\text{-}\divisible}$.
By (1), $\phi(L)[\mathfrak{p}_1^\infty]$ is infinite and the negation of (c) holds.

Now we prove that (b) implies (a) if $K$ is DKF and $L / K$ is Galois.
We use again the proof by contraposition.
Suppose that there is a nonzero $x \in \phi(L)_\divisible$.
For a nonzero $a \in A$, we set $X_a = \{y \in \phi(L) \mid x = a y\}$.
Then $(X_a)_a$ forms a projective system.
Since each $X_a$ is not empty by assumption, the projective limit $\varprojlim_a X_a$ is also not empty.
Let $(y_a)_a \in \varprojlim_a X_a$ and $K'$ a finite subextension of $L / K$ with $x \in \phi(K')$.
Then $K'$ is DKF since $K$ is so.
Thus there exists $a_0 \in A$ such that $y_{a_0} \notin \phi(K')$.
Let $\sigma_0$ be an element in $G_{K'}$ satisfying $y_{a_0}^{\sigma_0} \neq y_{a_0}$ and put $z_a = y_a^{\sigma_0} - y_a$ for each nonzero $a \in A$.
Then $z_{a_0} \neq 0$ and $a z_{a a_0} = z_{a_0}$ for each $a$.
Since $L / K$ is Galois, we have $z_{a_0} \in \phi(L)[a_0]$.
This shows that $z_{a_0}$ is a nonzero point in $(\phi(L)_\tor)_\divisible$, which establishes the negation of (b).
\end{proof}

In order to study DKF-ness in terms of ramification theory, we introduce the notion of the finiteness of the maximal breaks according to Ozeki--Taguchi~\cite[Definition~2.14]{OT}.
For any algebraic field extension $E$ over $F$, let $\mathfrak{M}_E^\fin$ be the set of primes of $E$ not lying above $\infty$ and we call the element of $\mathfrak{M}_E^\fin$ a \textit{finite prime}.
Let $\mathfrak{M}_E^\infty$ be the set of primes of $E$ lying above $\infty$ and we call the element of $\mathfrak{M}_E^\infty$ an \textit{infinite prime}.
Let $\mathfrak{M}_A$ be the set of nonzero prime ideals of $A$ and we identify it with $\mathfrak{M}^{\mathrm{f}}_F$.
Since, in contrast to the number field case, the influence of the ramification at infinite primes can not been essentially ignored in the function field case, the following definition takes into account such circumstances.

\begin{definition}\label{FMBOI}
\begin{enumerate}[(1)]
    \item Let $K$ be an algebraic extension of a local field $E$ of positive characteristic.
    Let $\widetilde{K}$ be the Galois closure of the maximal separable subextension of $K / E$.
    The \textit{maximal ramification break} (\textit{maximal break} for short) of $K / E$ is defined to be $\inf \{u \in [- 1, + \infty) \mid \Gal(\widetilde{K} / E)^u = 1\}$.
    Here, for $u \ge - 1$, we write $\Gal(\widetilde{K} / E)^u$ for the $u$-th upper ramification group of $\Gal(\widetilde{K} / E)$ in the sense of Serre~\cite[Chapter~IV, Section~3]{Serre}.
    The extension $K / E$ is said to have \textit{finite maximal break} if the maximal break of $K / E$ is finite.
    \item Let $K$ be a Galois extension of a function field $F$.
    For $\mathfrak{p} \in \mathfrak{M}_A$, the extension $K / F$ is said to have \textit{finite maximal break at $\mathfrak{p}$} if the extension $K_\mathfrak{P} / F_\mathfrak{p}$ has finite maximal break, where $\mathfrak{P} \in \mathfrak{M}_K^\fin$ is above $\mathfrak{p}$.
    Notice that this is defined independently of the choice of $\mathfrak{P}$.
    \item Let $K$ be a Galois extension of a function field $F$.
    The extension $K / F$ is said to have \textit{finite maximal break outside $\infty$} if $K / F$ has finite maximal break at every $\mathfrak{p} \in \mathfrak{M}_A$.
\end{enumerate}
\end{definition}

In the rest of this section, we assume $F = \mathbb{F}_q(t)$, $\infty = (1 / t)$, and $A = \mathbb{F}_q[t]$.

\begin{proposition}\label{FMBOItoTF}
    Let $K$ be a finite extension of $F$, $\phi$ a Drinfeld $A$-module over $K$, and $\mathfrak{p}$ a nonzero prime ideal of $A$.
    For a Galois extension $L / K$ with finite maximal break outside $\infty$, we have $T_\mathfrak{p}(\phi)^{G_L} = 0$.
\end{proposition}

\begin{proof}
Assume $T = T_\mathfrak{p}(\phi)^{G_L} \neq 0$.
The Galois group $G = \Gal(L / K)$ acts on $T$ and the representation arising from this is unramified at finite primes of $K$ not lying above $\mathfrak{p}$ at which $\phi$ has good reduction~\cite[Theorem~1]{Takahashi}.
Hence there exists a finite subset $S \subseteq \mathfrak{M}_K^\fin$ such that $T_\mathfrak{p}(\phi)$ is unramified outside $S \cup \mathfrak{M}_K^\infty$.
Let $\chi: G \to A_\mathfrak{p}^\times$ be the character induced from the action of $G$ on $\det T \cong A_\mathfrak{p}$ and $K^\dagger$ the intermediate field of $L / K$ corresponding to $\ker \chi$.
Then $K^\dagger / K$ is abelian.
For any $\mathfrak{l} \in S$, by~\cite[Chapter~IV, Corollary of Theorem~6.2]{FV}, the reciprocity map $K_\mathfrak{l} \to \Gal(K_\mathfrak{l}^\ab / K_\mathfrak{l})$ sends the $u$-th higher group $U_{K_\mathfrak{l}}^{(u)}$ of units of $K_\mathfrak{l}$ isomorphically onto $\Gal(K_\mathfrak{l}^\ab / K_\mathfrak{l})^u$ for any integer $u \ge 0$.
Composing with the natural projection $\Gal(K_\mathfrak{l}^\ab / K_\mathfrak{l}) \to \Gal(K^\dagger_\mathfrak{L} / K_\mathfrak{l})$, where $\mathfrak{L}$ denotes a prime of $K^\dagger$ lying over $\mathfrak{l}$, we have the surjection $U_{K_\mathfrak{l}}^{(u)} \to \Gal(K^\dagger_\mathfrak{L} / K_\mathfrak{l})^u$.
Then we obtain the surjection
\[ U_{K_\mathfrak{l}}^{(0)} / U_{K_\mathfrak{l}}^{(u)} \to \Gal(K^\dagger_\mathfrak{L} / K_\mathfrak{l})^{0} / \Gal(K^\dagger_\mathfrak{L} / K_\mathfrak{l})^u \]
for $u \ge 0$.
By assumption, $\Gal(K^\dagger_\mathfrak{L} / K_\mathfrak{l})^u$ is trivial for sufficiently large $u$.
Since $U_{K_\mathfrak{l}}^{(0)} / U_{K_\mathfrak{l}}^{(u)}$ is finite, we know that $\Gal(K^\dagger_\mathfrak{L} / K_\mathfrak{l})^0$ is finite.
Hence there exists a finite extension $K(\mathfrak{l})$ over $K$ in $K^\dagger$ such that the restriction of $\chi$ to $\Gal(L / K(\mathfrak{l}))$ is unramified at all primes of $K(\mathfrak{l})$ lying over $\mathfrak{l}$.

For any $\mathfrak{l} \in \mathfrak{M}_K^\infty$, let $\Lambda$ be the lattice in $K_\mathfrak{l}^\separable$ associated to $\phi$.
Then there exists a finite separable extension $E$ of $K_\mathfrak{l}$ with $\Lambda \subseteq E$.
The action of $G_K$ on $T_\mathfrak{p}(\phi)$ induces the action of $G_{K_\mathfrak{l}}$ on $\Lambda \otimes_A A_\mathfrak{p}$, under which $G_E$ acts trivially.
We set $K(\mathfrak{l}) = E$.
Then the restriction of $\chi$ to $\Gal(L / K(\mathfrak{l}))$ is unramified at all primes of $K(\mathfrak{l})$ lying over $\mathfrak{l}$.
We replace $K$ with the compositum of $K(\mathfrak{l})$ for all $\mathfrak{l} \in S \cup \mathfrak{M}_{K}^\infty$ and may assume that $\chi$ is unramified at every prime of $K$.

Take $\mathfrak{l} \in \mathfrak{M}_K^\fin$ not lying above $S$.
Then the eigenvalue $\alpha_\mathfrak{l}$ of the action of the Frobenius element on $T_\mathfrak{p}(\phi)$ satisfies $\lvert\alpha_\mathfrak{l}\rvert_\infty = q_\mathfrak{l}^{1 / r_\phi}$, where $q_\mathfrak{l}$ is the cardinality of the residue field at $\mathfrak{l}$ and $r_\phi$ is the rank of $\phi$~\cite[Proposition~3]{Takahashi}.
In particular, $\lvert\alpha_\mathfrak{l}\rvert_\infty > 1$.
Thus the image $\chi(G)$ of $G$ under the above $\chi$ is infinite.
Let $M$ be the extension of $K$ satisfying $\chi(G) = \Gal(M / K)$.
Then $M$ is an infinite abelian extension over $K$ unramified at every prime.
Since the constant extension in $K(T_\mathfrak{p}(\phi)) / K$ is finite~\cite[Remark~4.2]{Gekeler}, there exists a finite constant extension $K' / K$ in $M$ such that $M / K'$ is geometric and infinite abelian extension unramified at every prime. 
However, this is impossible since the maximal geometric extension unramified at every prime must be finite (see~\cite[Chapter~VIII, Section~3]{AT}).
\end{proof}

\begin{theorem}\label{FMBOItoDKF}
    A Galois extension $L / K$ with finite maximal break outside $\infty$ is DKF.
\end{theorem}

\begin{proof}
Let $L'$ be a finite Galois extension of $L$ and $\phi$ a Drinfeld $A$-module over $L'$.
There exists a finite subextension $K'$ of $L' / K$ such that $L' / K'$ is Galois and that $\phi$ is defined over $K'$.
From the assumption, $L' / K'$ has finite maximal break outside $\infty$.
By Proposition~\ref{FMBOItoTF}, for any nonzero prime ideal $\mathfrak{p}$ in $A$, we have $T_\mathfrak{p}(\phi)^{G_{L'}} = 0$.
Hence $\phi(L')[\mathfrak{p}^\infty]$ is finite by Proposition~\ref{criteria}~(1).
Now $K'$ is DKF as $K$ is so.
Then $\phi(L')_\divisible = 0$ by (2) of the same proposition.
Therefore $L$ is DKF.
\end{proof}

\begin{corollary}\label{ConstExt}
Let $L$ be as in Theorem~\emph{\ref{FMBOItoDKF}}.
Then $L \overline{\mathbb{F}_q}$ is DKF.
\end{corollary}

\begin{proof}
We see that $L \overline{\mathbb{F}_q} / K$ has finite maximal break outside $\infty$ and apply Theorem~\ref{FMBOItoDKF}.
\end{proof}

\section{Construction of Drinfeld-Kummer-faithful fields from Drinfeld modules}

Recall that $\mathfrak{M}^{\mathrm{f}}_E$ denotes the set of primes of $E$ not lying above $\infty$ for any finite field extension $E$ over $F$.
For $\mathfrak{l} \in \mathfrak{M}^{\mathrm{f}}_E$, denote by $v_\mathfrak{l}$ the valuation corresponding to $\mathfrak{l}$ normalized so that $v_\mathfrak{l}(E^\times) = \mathbb{Z}$.
We uniquely extend $v_\mathfrak{l}$ to $E_\mathfrak{l}^\separable$.
We use the notation $v_L$ for the valuation corresponding to a finite extension $L$ of $E_\mathfrak{l}$ normalized so that $v_L(L^\times) = \mathbb{Z}$.
Denote by $\mathcal{O}_\mathfrak{l}$ the valuation ring corresponding to $E_\mathfrak{l}$.
To simplify the notation, for a principal ideal $\mathfrak{a}$ of $A$, we will write again $\mathfrak{a}$ for its generator when the choice of the generator poses no problem.

The aim of this section is to give some examples of DKF fields which are infinitely generated over $F$.
We will show that the field obtained by adjoining the torsion points, of bounded order, of Drinfeld modules of bounded rank with ``not too bad" reduction, by referring to the strategy in Section~5 of the paper of Rosen~\cite{Rosen03}.
Because we are only interested in the finiteness of the maximal breaks, our estimate in this section will seem far from being optimal.
For more accurate estimates, see~\cite{Taguchi92} or~\cite{CL13EIT}.

\begin{proposition}\label{prop1}
Let $K$ be a finite extension of $F$.
Let $\mathfrak{p} \in \mathfrak{M}_A$ and $\mathfrak{l} \in \mathfrak{M}^{\mathrm{f}}_K$.
Let $m$, $r$, and $N$ be non-negative integers.
Assume that $\mathfrak{p}^m$ is a principal ideal.
Let $a(\mathfrak{p}^m)$ be the coefficient of the initial term of $\phi_{\mathfrak{p}^m}(X)$ for a Drinfeld $A$-module $\phi$ over $K_\mathfrak{l}$.
Then there is a constant $C$ depending on $\mathfrak{p}$, $\mathfrak{l}$, $m$, $r$, and $N$ such that, for any Drinfeld $A$-module $\phi$ over $\mathcal{O}_\mathfrak{l}$ of rank at most $r$ with $v_\mathfrak{l}(a(\mathfrak{p}^m)) \le N$, the maximal break of $K_\mathfrak{l}(\phi[\mathfrak{p}^m]) / K_\mathfrak{l}$ is at most $C$.
\end{proposition}

\begin{proof}
Looking at the Newton polygon of $\phi_{\mathfrak{p}^m}(X)$ at $\mathfrak{l}$, for a nonzero $x \in \phi[\mathfrak{p}^m]$, we have
\[ v_\mathfrak{l}(x) \le - \frac{0 - v_\mathfrak{l}(\mathfrak{p}^m)}{q - 1} = \frac{v_\mathfrak{l}(\mathfrak{p}^m)}{q - 1} \]
and
\[ v_\mathfrak{l}(x) \ge - \frac{v_\mathfrak{l}(a(\mathfrak{p}^m)) - 0}{q^{r_\phi \deg(\mathfrak{p}^m)} - q^{r_\phi \deg(\mathfrak{p}^m) - 1}} = \frac{- v_\mathfrak{l}(a(\mathfrak{p}^m))}{q^{r_\phi \deg(\mathfrak{p}^m) - 1} (q - 1)}. \]
Here $r_\phi$ denotes the rank of $\phi$.
Suppose that $G_{K_\mathfrak{l}}^u \neq 1$.
For $\sigma \neq 1$ in $G_{K_\mathfrak{l}}^u$, there exists $x$ in $\phi[\mathfrak{p}^m]$ such that $x^\sigma \neq x$.
Take $b \in K_\mathfrak{l}$ with $- v_\mathfrak{l}(x) \le v_\mathfrak{l}(b) < - v_\mathfrak{l}(x) +1$.
Then $v_\mathfrak{l}(b x) \ge 0$ and
\begin{align*}
    v_L((b x)^\sigma - b x) &= e_{L / K_\mathfrak{l}} v_\mathfrak{l}((b x)^\sigma - b x) \\
    &= e_{L / K_\mathfrak{l}} (v_\mathfrak{l}(x^\sigma - x) + v_\mathfrak{l}(b)) \\
    &\le e_{L / K_\mathfrak{l}} \left( \frac{v_\mathfrak{l}(\mathfrak{p}^m)}{q - 1} + \left(- \frac{- v_\mathfrak{l}(a(\mathfrak{p}^m))}{q^{r_\phi \deg(\mathfrak{p}^m) - 1} (q - 1)} \right) +1 \right).
\end{align*}
Here $L = K_\mathfrak{l}(\phi[\mathfrak{p}^m])$ and $e_{L / K_\mathfrak{l}}$ denotes the ramification index of $L / K_\mathfrak{l}$.
Thus
\[ u \le  e_{L / K_\mathfrak{l}} \left( \frac{v_\mathfrak{l}(\mathfrak{p}^m)}{q - 1} + \frac{v_\mathfrak{l}(a(\mathfrak{p}^m))}{q^{r_\phi \deg(\mathfrak{p}^m) - 1} (q - 1)} +1 \right) - 1.\]
Since $e_{L / K_\mathfrak{l}}$ is bounded in terms of $\mathfrak{p}$, $m$, and $r$ (An obvious bound is $(q^{r \deg(\mathfrak{p}^m)})!$ because $L / K_\mathfrak{l}$ is the splitting field of the polynomial $\phi_{\mathfrak{p}^m}(X)$, whose degree is at most $q^{r \deg(\mathfrak{p}^m)}$), we obtain a desired upper bound of $u$.
\end{proof}

From now on, we assume $F = \mathbb{F}_q(t)$, $\infty = (1 / t)$, and $A = \mathbb{F}_q[t]$.
For a Drinfeld $A$-module $\phi$ over $K$ with stable reduction, let $(\psi, \Gamma)$ be the Tate uniformization of $\phi$ at $\mathfrak{l} \in \mathfrak{M}^{\mathrm{f}}_K$.
Here $\psi$ is a Drinfeld $A$-module over $\mathcal{O}_\mathfrak{l}$ with good reduction and $\Gamma$ is a $\psi$-lattice, which means a finitely generated discrete projective $A$-submodule of $\psi(K_\mathfrak{l}^\separable)$ stable under the action of $G_{K_\mathfrak{l}}$.
Such a pair is called a \textit{Tate datum}.
The power series
\[ e_\Gamma(X) = X \prod_{\gamma \in \Gamma \setminus \{0\}} \left(1 - \frac{X}{\gamma}\right) \]
defines an $\mathbb{F}_q$-linear entire function on $\psi(K_\mathfrak{l}^\separable)$ and satisfies $e_\Gamma \psi_a = \phi_a e_\Gamma$ for all $a \in A$ (we can see $e_\Gamma$ as an element in $\mathcal{O}_\mathfrak{l}\{\{\tau\}\}$, the non-commutative ring of formal power series in $\tau$ with coefficients in $\mathcal{O}_\mathfrak{l}$, whose multiplication is determined by $\tau x = x^q \tau$ for $x \in \mathcal{O}_\mathfrak{l}$).
For positive integers $r_\psi$ and $r_\Gamma$, there is a natural one-to-one correspondence (see~\cite[Section 7]{Drinfeld})
\begin{align*}
    &\left\{ \begin{array}{l}
        \text{Drinfeld $A$-modules } \phi \text{ of rank }  r_\psi + r_\Gamma \text{ over} \\
         K_\mathfrak{l} \text{ with stable reduction of rank } r_\psi
    \end{array} \right\} \bigg/ (K_\mathfrak{l}\text{-isom.}) \\
    &\stackrel{1 : 1}{\longleftrightarrow} \left\{ \begin{array}{l}
        \text{Tate data } (\psi, \Gamma), \text{ where } \psi \text{ is a Drinfeld}\\
        \text{$A$-module over } \mathcal{O}_\mathfrak{l} \text{ with good reduction of} \\
        \text{rank } r_\psi \text{ and } \Gamma \text{ is a $\psi$-lattice of rank } r_\Gamma
    \end{array} \right\} \Bigg/ (K_\mathfrak{l}\text{-isom.}).
\end{align*}
Note that two Tate data $(\psi, \Gamma)$ and $(\psi', \Gamma')$ are $K_\mathfrak{l}$-isomorphic if there exists an isomorphism from $\psi$ to $\psi'$ over $K_\mathfrak{l}$ which induces an isomorphism from $\Gamma$ to $\Gamma'$.

We introduce the notion of the size of a Tate datum $(\psi, \Gamma)$.
We notice that, as we have assumed $A = \mathbb{F}_q[t]$, any $\psi$-lattice $\Gamma$ is a finitely generated free $A$-submodule in $\psi(K_\mathfrak{l}^\separable)$.
Since $\Gamma$ is discrete and $\psi$ has good reduction, we have $v_\mathfrak{l}(\gamma) < 0$ and $v_\mathfrak{l}(\psi_a(\gamma)) = q^{r_\psi \deg(a)} v_\mathfrak{l}(\gamma)$ for $\gamma \in \Gamma \setminus \{ 0 \}$ and $a \in A \setminus \{ 0 \}$.
According to Gardeyn~\cite[Section~1]{Gardeyn}, we define
\[ \lVert\gamma\rVert_\mathfrak{l} = (- v_{\mathfrak{l}_0}(\gamma))^{1 / r_\psi} \]
for $\gamma \in \Gamma \setminus \{0\}$ and $\lVert0\rVert_\mathfrak{l} = 0$.
Here $\mathfrak{l}_0 \in \mathfrak{M}_A$ lies under $\mathfrak{l}$.
Then $\lVert-\rVert_\mathfrak{l}$ is a $G_{K_\mathfrak{l}}$-invariant norm relative to the absolute value $\lvert-\rvert_\infty = q^{\deg(-)}$ at $\infty$ and is unchanged under a finite extension of $K$.
We recall the notion of successive minima of an $A$-lattice from the paper of Taguchi~\cite[Section~4]{Taguchi93} in a general setting.
Notice that an $A$-lattice means a finitely generated discrete free $A$-module with respect to some norm $\lVert-\rVert$.

\begin{definition}
Let $\Lambda$ be an $A$-lattice with respect to a norm $\lVert-\rVert$.
For a real number $c$, we set
\[ B(c) = \{ \lambda \in \Lambda \mid \lVert\lambda\rVert \le c\}. \]
Let $r$ be the rank of $\Lambda$.
For $1 \le i \le r$, the \textit{$i$-th successive minimum} $c_i$ is the smallest real number $c$ such that $B(c)$ contains at least $i$ elements of $\Lambda$ which are linearly independent over $A$.
An $A$-basis $(\lambda_i)_{1 \le i \le r}$ of $\Lambda$ is said to be a \textit{successive minimum basis} if it satisfies $\lVert\lambda_i\rVert = c_i$ for all $i$.
\end{definition}

Note that the set $B(c)$ is finite for any real number $c$.
We easily verify that the successive minima and the successive minimum basis always exist.
A useful property of successive minima is the next lemma.

\begin{lemma}[{\cite[Lemma~4.2]{Taguchi93}}]
Let $\Lambda$ be an $A$-lattice of rank $r$ with respect to a norm $\lVert-\rVert$ and $(\lambda_i)_{1 \le i \le r}$ an $A$-basis of $\Lambda$ such that $\lVert\lambda_1\rVert \le \dotsb \le \lVert\lambda_r\rVert$.
Then the following conditions are equivalent$:$
\begin{enumerate}[\textup{(\alph{enumi})}]
    \item $(\lambda_i)_{1 \le i \le r}$ is a successive minimum basis.
    \item $\left\lVert\sum a_i \lambda_i\right\rVert = \max \{\lVert a_i \lambda_i \rVert\}$ for all $(a_i)_{1 \le i \le r} \in A^r$.
\end{enumerate}
\end{lemma}

For an $A$-lattice $\Lambda$ with successive minima $c_1, \dotsc, c_r$, we define the \textit{covolume} $D(\Lambda)$ of $\Lambda$ as their product, i.e.,
\begin{equation*}
    D(\Lambda) = \prod_{i = 1}^r c_i.
\end{equation*}

The next lemma is easy.

\begin{lemma}\label{EstimationOfTheLastSM}
Let $\Lambda$ be an $A$-lattice of rank $r$ and $(c_i)_{1 \le i \le r}$ its successive minima.
If $\varepsilon$ is a real positive number with $\varepsilon \le c_1$, then we have $c_r \le \varepsilon^{- (r - 1)} D(\Lambda)$.
\end{lemma}

\begin{proof}
Since $D(\Lambda) \ge c_1^{r - 1} c_r$, we have $c_r \le c_1^{- (r - 1)} D(\Lambda) \le \varepsilon^{- (r - 1)} D(\Lambda)$.
\end{proof}

Now we turn to the situation of a Drinfeld $A$-module $\phi$ of rank at most $r$ over $K$.
From~\cite[Proposition~7.1]{Drinfeld}, there is a tamely ramified extension $K'$ of $K$ of ramification index dividing $d = \mathrm{lcm}\{q^i - 1 \mid 1 \le i \le r\}$ such that $\phi$ has stable reduction at $\mathfrak{l}' \in \mathfrak{M}_{K'}^\fin$ lying above $\mathfrak{l}$.
Let $(\psi, \Gamma)$ be the Tate uniformization at $\mathfrak{l}'$.
Recall that $\Gamma$ is an $A$-lattice under the action of $A$ via $\psi$ and the norm $\lVert-\rVert = \lVert-\rVert_{\mathfrak{l}'}$.
Put $D(\phi, \mathfrak{l}) = D(\Gamma)$.
Notice that it is independent of the choices of an extension $K'$, a prime $\mathfrak{l}'$, and a Tate datum $(\psi, \Gamma)$.
We also remark that, for a Galois extension $L / K$, the maximal break of $L K' / K'$ is at most $d$ times that of $L / K$.
We write $r_\phi$ for the rank of $\phi$, $r_\psi$ for the rank of $\psi$, and $r_\Gamma$ for the rank of $\Gamma$.

\begin{proposition}\label{prop2}
Let $K$ be a finite extension of $F$.
Let $\mathfrak{p} \in \mathfrak{M}_A$ and $\mathfrak{l} \in \mathfrak{M}^{\mathrm{f}}_K$.
Let $m$, $r$, and $N$ be non-negative integers.
Then there is a constant $C$ depending on $\mathfrak{p}$, $\mathfrak{l}$, $m$, $r$, and $N$ such that, for any Drinfeld $A$-module $\phi$ over $\mathcal{O}_\mathfrak{l}$ of rank at most $r$ with $D(\phi, \mathfrak{l}) \le N$, the maximal break of $K_\mathfrak{l}(\phi[\mathfrak{p}^m]) / K_\mathfrak{l}$ is at most $C$.
\end{proposition}

\begin{proof}
By Proposition \ref{prop1}, it suffices to give an upper bound of $v_\mathfrak{l}(a(\mathfrak{p}^m))$ in terms of $\mathfrak{p}$, $\mathfrak{l}$, $m$, $r$, and $N$.
The above remark allows us to replace $K$ with its tamely ramified extension of ramification index dividing $d$ and to assume that $\phi$ has stable reduction at $\mathfrak{l}$ over $K$.
Let $(\psi, \Gamma)$ be a Tate uniformization of $\phi$ at $\mathfrak{l}$.
If $r_\psi = r_\phi$, then $\phi$ has good reduction at $\mathfrak{l}$ and $v_\mathfrak{l}(a(\mathfrak{p}^m)) = 0$.
We assume $r_\psi < r_\phi$ for the rest of the proof.
Let $(\gamma_i)_{1 \le i \le r_\Gamma}$ be a successive minimum basis of $\Gamma$.
We have a natural isomorphism
\[ \psi_{\mathfrak{p}^m}^{- 1}(\Gamma) / \Gamma \stackrel{\sim}{\to} \phi[\mathfrak{p}^m]; \quad z + \Gamma \mapsto e_\Gamma(z). \]
Then
\[ \phi_{\mathfrak{p}^m}(X) = \mathfrak{p}^m X \prod_{\xi \in \phi[\mathfrak{p}^m] \setminus \{0\}} \left(1 - \frac{X}{\xi}\right) = \mathfrak{p}^m X \prod_{z \in (\psi_{\mathfrak{p}^m}^{- 1}(\Gamma) / \Gamma) \setminus \{0\}} \left(1 - \frac{X}{e_\Gamma(z)}\right). \]
Comparing the initial coefficient of both hand sides, we obtain
\[ a(\mathfrak{p}^m) = \mathfrak{p}^m \Big/ \prod e_\Gamma(z), \]
where $z$ runs over $(\psi_{\mathfrak{p}^m}^{- 1}(\Gamma) / \Gamma) \setminus \{0\}$.

Let $\xi_i$ be a solution of the equation $\psi_{\mathfrak{p}^m}(X) = \gamma_i$.
We claim that the set
\[ Z = \left\{\sum_{i = 1}^{r_\Gamma} \psi_{a_i}(\gamma_i) + \eta \mathrel{}\middle|\mathrel{} a_i \in A \text{ with } \deg(a_i) < \deg(\mathfrak{p}^m), \text{ and } \eta \in \psi[\mathfrak{p}^m]\right\} \]
is a system of representatives for $\psi_{\mathfrak{p}^m}^{- 1}(\Gamma) / \Gamma$ with $0 \in Z$.
Indeed, we can easily confirm that $Z$ is contained in $\psi_{\mathfrak{p}^m}^{- 1}(\Gamma)$, that its cardinality is equal to $q^{r_\phi \deg(\mathfrak{p}^m)} = \# \phi[\mathfrak{p}^m]$, and that no two elements in $Z$ are congruent modulo $\Gamma$.

In order to obtain an upper bound of $v_\mathfrak{l}(a(\mathfrak{p}^m))$, we should find a lower bound of $v_\mathfrak{l}(e_\Gamma(z))$ for a nonzero $z \in Z$.
We have
\begin{align*}
    v_\mathfrak{l}(e_\Gamma(z)) &= v_\mathfrak{l}\left( z \prod_{\gamma \in \Gamma \setminus \{0\}}\left(1 - \frac{z}{\gamma}\right)\right) \\
    &\ge v_\mathfrak{l}(z) + \sum_{\gamma \in \Gamma \setminus \{0\}; \, v_\mathfrak{l}(\gamma) > v_\mathfrak{l}(z)} v_\mathfrak{l}\left(\frac{z}{\gamma}\right) \\
    &\ge v_\mathfrak{l}(z) \cdot \# \{\gamma \in \Gamma \mid v_\mathfrak{l}(\gamma) > v_\mathfrak{l}(z)\}.
\end{align*}
The last inequality follows since $v_\mathfrak{l}(\gamma) < 0$ for a nonzero $\gamma \in \Gamma$.
For a nonzero $z = \sum \psi_{a_i}(\xi_i) + \eta \in Z$, we evaluate
\begin{align*}
    v_\mathfrak{l}(z) &\ge \min(\{v_\mathfrak{l}(\psi_{a_i}(\xi_i))\}_i \cup \{v_\mathfrak{l}(\eta)\}) \\
    &= \min\{q^{r_\psi \deg(a_i)} v_\mathfrak{l}(\xi_i)\}_i \\
    &\ge q^{r_\psi (\deg(\mathfrak{p}^m) - 1)} \cdot q^{- r_\psi \deg(\mathfrak{p}^m)} \min\{v_\mathfrak{l}(\gamma_i)\}_i \\
    &= - q^{- r_\psi} c_{r_\Gamma}^{r_\psi}.
\end{align*}
Then we find
\begin{align*}
    & \# \{\gamma \in \Gamma \mid v_\mathfrak{l}(\gamma) > v_\mathfrak{l}(z)\} \\
    &\le \# \{\gamma \in \Gamma \mid v_\mathfrak{l}(\gamma) > - q^{- r_\psi} c_{r_\Gamma}^{r_\psi}\} \\
    &= \# \left\{(a_1, \dotsc, a_{r_\Gamma}) \in A^{r_\Gamma} \mathrel{}\middle|\mathrel{} v_\mathfrak{l}\left(\sum_{i = 1}^{r_\Gamma} \psi_{a_i}(\gamma_i)\right) > - q^{- r_\psi} c_{r_\Gamma}^{r_\psi}\right\} \\
    &= \prod_{i = 1}^{r_\Gamma} \# \{a_i \in A \mid v_\mathfrak{l}(\psi_{a_i}(\gamma_i)) > - q^{- r_\psi} c_{r_\Gamma}^{r_\psi}\} \\
    &= \prod_{i = 1}^{r_\Gamma} \# \left\{a_i \in A \mathrel{}\middle|\mathrel{} \deg(a_i) < \log_q \left(\frac{c_{r_\Gamma}}{c_i}\right) - 1\right\} \\
    &< \prod_{i = 1}^{r_\Gamma} \frac{c_{r_\Gamma}}{c_i}.
\end{align*}

On the other hand, the conjugates of $\gamma_1$ are contained in the set $\{\gamma \in \Gamma \mid \lVert\gamma\rVert_\mathfrak{l} = c_1\}$, whose cardinality is at most $q^{r_\Gamma}-1$.
Hence the degree of the splitting field $L'$ of the minimal polynomial of $\gamma_1$ over $K_\mathfrak{l}$ is bounded and there exists a constant $C'$ depending on $r$ such that the ramification index $e_{L' / K_\mathfrak{l}}$ of $L' / K_\mathfrak{l}$ is at most $C'$.
(An obvious bound is $C' = (q^r - 1)!$.)
Since $v_\mathfrak{l}(\gamma_1) < 0$, we have $v_\mathfrak{l}(\gamma_1) \le - 1 / C'$ so $c_1 = (- v_\mathfrak{l}(\gamma_1))^{1 / r_\psi} \ge C'^{- 1 / r_\psi}$.
Applying Lemma~\ref{EstimationOfTheLastSM} with $\varepsilon = C'^{- 1 / r_\psi}$, we obtain
\begin{align*}
    v_\mathfrak{l}(a(\mathfrak{p}^m)) &= v_\mathfrak{l}(\mathfrak{p}^m) - \sum_{z \in Z \setminus \{0\}} v_\mathfrak{l}(e_\Gamma(z)) \\
    &\le v_\mathfrak{l}(\mathfrak{p}^m) - (q^{r_\phi \deg(\mathfrak{p}^m)} - 1) (- q^{- r_\psi} c_{r_\Gamma}^{r_\psi}) \prod_{i = 1}^{r_\Gamma} \frac{c_{r_\Gamma}}{c_i} \\
    &< v_\mathfrak{l}(\mathfrak{p}^m) + q^{r_\phi \deg(\mathfrak{p}^m) - 1} c_{r_\Gamma}^{r_\phi} \prod_{i = 1}^{r_\Gamma - 1} c_i^{- 1} \\
    &\le v_\mathfrak{l}(\mathfrak{p}^m) + q^{r_\phi \deg(\mathfrak{p}^m) - 1} (C'^{(r_\Gamma - 1) / r_\psi} D(\phi, \mathfrak{l}))^{r_\phi} C'^{r_\phi - 2} \\
    &\le v_\mathfrak{l}(\mathfrak{p}^m) + q^{r_\phi \deg(\mathfrak{p}^m) - 1} N^{r_\phi} C'^{(r_\phi + 1)(r_\phi - 2)},
\end{align*}
which proves the proposition.
\end{proof}

For a positive integer $r$, a family of positive integers $\boldsymbol{N} = (N_\mathfrak{l})_{\mathfrak{l} \in \mathfrak{M}_K^\fin}$ indexed by $\mathfrak{M}_K^\fin$, and a family of non-negative integers $\boldsymbol{m} = (m_\mathfrak{p})_{\mathfrak{p} \in \mathfrak{M}_A}$ indexed by $\mathfrak{M}_A$, let $\Phi_K(r, \boldsymbol{N}, \boldsymbol{m})$ be the set of Drinfeld $A$-module $\phi$ over $K$ of rank at most $r$ with $D(\phi, \mathfrak{l}) \le N_\mathfrak{l}$ for all $\mathfrak{l} \in \mathfrak{M}_K^\fin$.
Define $K_{r, \boldsymbol{N}, \boldsymbol{m}}$ to be the extension of $K$ generated by all elements of $\phi[\mathfrak{p}^{m_\mathfrak{p}}]$ for all $\phi \in \Phi_K(r, \boldsymbol{N}, \boldsymbol{m})$ and all $\mathfrak{p} \in \mathfrak{M}_A$.

\begin{theorem}\label{mythm}
The field $K_{r, \boldsymbol{N}, \boldsymbol{m}}$ is DKF.
\end{theorem}

\begin{proof}
By Theorem~\ref{FMBOItoDKF}, it is enough to show that $K_{r, \boldsymbol{N}, \boldsymbol{m}}$ has finite maximal break outside $\infty$.
We remark that Lemma~3.1 in \cite{OT} still holds in the context of function fields.
Thus it suffices to prove that, for any $\mathfrak{l} \in \mathfrak{M}_K^\fin$, there exists a constant $u(\mathfrak{l})$ such that the maximal break of $K_\mathfrak{l}(\phi[\mathfrak{p}^{m_\mathfrak{p}}]) / K_\mathfrak{l}$ is at most $u(\mathfrak{l})$ for any $\phi \in \Phi_K(r, \boldsymbol{N}, \boldsymbol{m})$ and any $\mathfrak{p} \in \mathfrak{M}_A$.

Fix $\mathfrak{l} \in \mathfrak{M}_K^\fin$, $\phi \in \Phi_K(r, \boldsymbol{N}, \boldsymbol{m})$, and $\mathfrak{p} \in \mathfrak{M}_A$.
By the remark preceding Proposition~\ref{prop2}, we may assume that $\phi$ has stable reduction at $\mathfrak{l}$ over $K$.

If $\mathfrak{p}$ lies beneath $\mathfrak{l}$, then there is a constant $C_1$ such that the maximal break of $K_\mathfrak{l}(\phi[\mathfrak{p}^{m_\mathfrak{p}}]) / K_\mathfrak{l}$ is at most $C_1$ by Proposition \ref{prop2}.
Suppose $\mathfrak{l} \nmid \mathfrak{p}$.
Let $(\psi, \Gamma)$ be the Tate uniformization of $\phi$ at $\mathfrak{l}$.
Since $\psi$ has good reduction, $K_\mathfrak{l}(\psi[\mathfrak{p}^{m_\mathfrak{p}}]) / K_\mathfrak{l}$ is unramified~\cite[Theorem~1]{Takahashi}.
Let $u$ be a real number with $G_{K_\mathfrak{l}}^u \neq 1$ and $\sigma \neq 1$ belong to $G_{K_\mathfrak{l}}^u$.
Then there is $i$ with $\gamma_i^\sigma \neq \gamma_i$.
From the definition of the first successive minimum $c_1$, we have $\lVert\gamma_i^\sigma - \gamma_i\rVert_\mathfrak{l} \ge c_1$.
Hence $v_\mathfrak{l}(\gamma_i^\sigma - \gamma_i) \le - c_1^{r_\psi}$.
Take $b \in K_\mathfrak{l}$ with $- v_\mathfrak{l}(\gamma_i) \le v_\mathfrak{l}(b) < - v_\mathfrak{l}(\gamma_i) + 1$.
Then $v_\mathfrak{l}(b \gamma_i) \ge 0$ and
\begin{align*}
    v_{K_\mathfrak{l}(\Gamma)}((b \gamma_i)^\sigma - b \gamma_i) &= e_{K_\mathfrak{l}(\Gamma) / K_\mathfrak{l}}v_\mathfrak{l}((b \gamma_i)^\sigma - b \gamma_i) \\
    &= e_{K_\mathfrak{l}(\Gamma) / K_\mathfrak{l}}(v_\mathfrak{l}(\gamma_i^\sigma - \gamma_i) + v_\mathfrak{l}(b)) \\
    &< e_{K_\mathfrak{l}(\Gamma) / K_\mathfrak{l}}(- c_1^{r_\psi} - v_\mathfrak{l}(\gamma_i) + 1) \\
    &\le e_{K_\mathfrak{l}(\Gamma) / K_\mathfrak{l}}(- c_1^{r_\psi} + c_{r_\Gamma}^{r_\psi} + 1),
\end{align*}
where $e_{K_\mathfrak{l}(\Gamma) / K_\mathfrak{l}}$ is the ramification index of $K_\mathfrak{l}(\Gamma) / K_\mathfrak{l}$.
Thus
\[ u \le e_{K_\mathfrak{l}(\Gamma) / K_\mathfrak{l}}(- c_1^{r_\psi} + c_{r_\Gamma}^{r_\psi} + 1) - 1. \]
By~\cite[Proposition~4]{Gardeyn}, on the different $\mathfrak{D}(K_\mathfrak{l}(\Gamma) / K_\mathfrak{l})$ of $K_\mathfrak{l}(\Gamma) / K_\mathfrak{l}$, we have
\[ \ord_\mathfrak{l} \mathfrak{D}(K_\mathfrak{l}(\Gamma) / K_\mathfrak{l}) \le 1 + 2 \sum_{i = 1}^{r_\Gamma} \left(q^{i - 1} v_\mathfrak{l}\left(\frac{\gamma_1}{\gamma_i}\right) \prod_{j = 1}^i \frac{c_i}{c_j}\right). \]
As in the proof of Proposition~\ref{prop2}, there exists a constant $C'$ depending only on $r$ such that $c_1 \ge C'^{- 1 / r_\psi}$.
Then by~\cite[Chapter~III, Section~6, Proposition~13]{Serre}, we evaluate
\begin{align*}
    e_{K_\mathfrak{l}(\Gamma) / K_\mathfrak{l}} &\le 1 + \ord_\mathfrak{l} \mathfrak{D}(K_\mathfrak{l}(\Gamma) / K_\mathfrak{l}) \\
    &= 2 + 2 \sum_{i = 1}^{r_\Gamma} \left(q^{i - 1} (- c_1^{r_\psi} + c_i^{r_\psi}) \prod_{j = 1}^i \frac{c_i}{c_j}\right) \\
    &\le 2 + 2 \sum_{i = 1}^{r_\Gamma} q^{i - 1} c_i^{r_\psi} \left(\frac{c_i}{c_1}\right)^{i - 1} \\
    &\le 2 + 2 c_{r_\Gamma}^{r_\psi} \sum_{i = 1}^{r_\Gamma} \left(\frac{q c_{r_\Gamma}}{c_1}\right)^{i - 1} \\
    &< 2 + 2 c_{r_\Gamma}^{r_\psi} \cdot 2 \left(\frac{q c_{r_\Gamma}}{c_1}\right)^{r_\Gamma - 1} \\
    &= 2 + 4 q^{r_\Gamma - 1} c_{r_\Gamma}^{r_\phi - 1} c_1^{- (r_\Gamma - 1)} \\
    &\le 2 + 4 q^{r - 2} (C'^{r - 2} D(\phi, \mathfrak{l}))^{r - 1} C'^{r - 2} \\
    &\le 2 + 4 q^{r - 2} C'^{r (r - 2)} N_\mathfrak{l}^{r - 1}.
\end{align*}
Here we use $\sum_{i = 1}^{r_\Gamma} \alpha^{i - 1} < 2 \alpha^{r_\Gamma - 1}$ for $\alpha \ge 2$, and apply Lemma~\ref{EstimationOfTheLastSM}.
Hence
\begin{align*}
    u &< (2 + 4 q^{r - 2} C'^{r (r - 2)} N_\mathfrak{l}^{r - 1}) ((C'^{r_\Gamma - 1} N_\mathfrak{l})^{r_\psi} + 1) - 1 \\
    &\le (2 + 4 q^{r - 2} C'^{r (r - 2)} N_\mathfrak{l}^{r - 1}) (C'^{(r - 1)^2 / 4} N_\mathfrak{l}^r + 1) - 1.
\end{align*}
Set $C_2$ to be the rightmost hand side in the above inequality.

Now we put $u(\mathfrak{l}) = \max \{C_1, C_2\}$.
We are going to show that this $u(\mathfrak{l})$ is the desired constant.
The case $\mathfrak{l} \mid \mathfrak{p}$ is clear.
Suppose that $\mathfrak{l} \nmid \mathfrak{p}$.
We have the exact sequence
\[ 0 \to \psi[\mathfrak{p}^{m_\mathfrak{p}}] \stackrel{e_\Gamma}{\to} \phi[\mathfrak{p}^{m_\mathfrak{p}}] \to \Gamma / \mathfrak{p}^{m_\mathfrak{p}} \Gamma \to 0 \]
on which $G_{K_\mathfrak{l}}$ acts compatibly.
We view this sequence as that of $\mathbb{F}_q$-vector spaces.
Then this sequence splits.
If $u \ge u(\mathfrak{l})$, then $G_{K_\mathfrak{l}}^u$ acts trivially on $\psi[\mathfrak{p}^{m_\mathfrak{p}}]$ and $\Gamma / \mathfrak{p}^{m_\mathfrak{p}} \Gamma$.
Therefore $G_{K_\mathfrak{l}}^u$ also acts trivially on $\phi[\mathfrak{p}^{m_\mathfrak{p}}]$ and the proof is completed.
\end{proof}

\begin{remark}
There exists a constant $C$ depending only on $K$ and $r$ such that, for any Drinfeld $A$-module $\phi$ over $K$ of rank $r$, the degree of the algebraic closure of $\mathbb{F}_q$ in $K(\phi(K^\separable)_\tor)$ is at most $C$~\cite[Remark~4.2]{Gekeler}.
Therefore, the algebraic closure of $\mathbb{F}_q$ in $K_{r, \boldsymbol{N}, \boldsymbol{m}}$ is finite.
From Corollary~\ref{ConstExt}, we see that $K_{r, \boldsymbol{N}, \boldsymbol{m}} \overline{\mathbb{F}_q}$ is an infinite extension of $K_{r, \boldsymbol{N}, \boldsymbol{m}}$ and still DKF.
\end{remark}

\bibliographystyle{amsalpha}

\end{document}